\pdfoutput=1
\documentclass[12pt]{amsart}
\usepackage{amssymb}
\usepackage{fullpage}
\usepackage{graphicx}
\usepackage{pictexwd,dcpic}
\usepackage{wrapfig}
\usepackage{enumitem}
\usepackage{url}

\theoremstyle{plain}
\newtheorem{thm}{Theorem}[]
\newtheorem{prop}[thm]{Proposition}

\theoremstyle{definition}
\newtheorem{defi}[thm]{Definition}
\theoremstyle{remark}
\newtheorem{rem}[thm]{Remark}

\newtheorem{exa}[thm]{Example}

\begin{document}
\title{$\text{q}$-holonomic formulas for colored HOMFLY polynomials of 2-bridge links}
\author{Paul Wedrich}
\address{Centre for Mathematical Sciences, University of Cambridge, CB3 0WB, England}
\email{P.Wedrich@dpmms.cam.ac.uk}
\begin{abstract}
We compute q-holonomic formulas for the HOMFLY polynomials of 2-bridge links colored with one-column (or one-row) Young diagrams. 
\end{abstract}
\maketitle
\setcounter{page}{1}
%
%
%

\section{Introduction and statement of results}
The colored HOMFLY polynomial is an invariant of framed, oriented links $L$ in $S^3$ whose components are colored with Young diagrams (or alternatively partitions of integers). It takes values in the ring $\mathbb{Z}[a^{\pm 1}](q)$. In this note we mostly consider colorings by one-column Young diagrams, i.e. partitions of the form $1^r:=\underbrace{(1,\dots, 1)}_{\text{$r$ $1$s}}$ and refer to them as $r$-colorings.

Let $L$ be a framed, oriented link with components numbered $1,\dots, k$.  Let $P_{c_1,\dots, c_k}(L)$ be the colored HOMFLY polynomial of $L$ with coloring $c_i$ on the $i^{th}$ component. As expected, this infinite set $\{P_{c_1,\dots, c_k}\mid c_i\in \mathbb{N}\}$  of invariants carries only a finite amount of information. More precisely, Garoufalidis \cite{Gar} proved that $P_{c_1,\dots, c_k}(L)$ is q-holonomic in the $c_i$. 

\begin{defi} A one-parameter sequence $(f_n\mid n\in \mathbb{N})$ of polynomials in $\mathbb{Z}[a^{\pm 1}](q)$ is \emph{q-holonomic} if there exist $d \in \mathbb{N}$ and polynomials $a_l\in \mathbb{Z}[u,v]$ for $0\leq l\leq d$ such that: 
\[ \sum_{l=0}^d a_l(q,q^n) f_{n+l} =0 \quad \text{for all }n\geq 0\]

\end{defi}
This recurrence relation can be encoded as a polynomial $A\in \mathbb{Z}[a^{\pm 1},q^{\pm 1}][M]\langle L \rangle$, with $L M = q M L$. Here $L$ and $M$ are considered as operators that act on the sequence $f_*$ by $(M f_n)(q)= q^n f_n(q)$ and $(L f_n)(q)=f_{n+1}(q)$. The recurrence relation then takes the compact form $A f_*=0$. We will not need the multivariable generalization of q-holonomic sequences \cite{Zei}, \cite{GL} in this note, see part (1) of Theorem \ref{mainthm}.\\

As intermediate step for computing these recurrence relations, it suffices to find a formula for the colored HOMFLY polynomial that is a multi-dimensional sum of q-proper hypergeometric summands. Then the recursion relation can be computed algorithmically \cite{WZe}, see \cite{Gar}. Explicit q-holonomic formulas are known for torus links \cite{BMS} and finitely many twist knots \cite{Kaw}. The aim of this note is to compute explicit q-holonomic formulas for the colored HOMFLY polynomials of 2-bridge links. 

\begin{defi} A 2-bridge link is a link $L$ such that the pair $(S^3,L)$ can be split by an embedded $(S^2, \{4 \text{ points}\})$ into two pairs that are both homeomorphic to $(D^2\times [0,1],\{ 2 \text{ points}\}\times [0,1])$. 
\end{defi}
A 2-bridge link has one or two components and in the latter case they are unknots. Every 2-bridge link has a special link diagram that can be constructed as follows:
\begin{enumerate}
\item Start with the trivial tangle $\vcenter{\hbox{\includegraphics[height=0.5cm, angle=0]{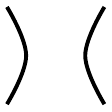}}}$.
\item Glue a finite number of crossings $\vcenter{\hbox{\includegraphics[height=0.5cm, angle=0]{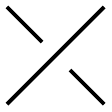}}}$ (or $\vcenter{\hbox{\includegraphics[height=0.5cm, angle=90]{Kcr.pdf}}}$) to the top endpoints.
\item Glue a finite number of crossings $\vcenter{\hbox{\includegraphics[height=0.5cm, angle=0]{Kcr.pdf}}}$ (or $\vcenter{\hbox{\includegraphics[height=0.5cm, angle=90]{Kcr.pdf}}}$) to the right endpoints.
\item Repeat (2) and (3) finitely many times. This produces a positive (or negative) rational tangle.
\item Close the rational tangle up by connecting the four endpoints by two arcs, without introducing new crossings and without making crossings nugatory.
\end{enumerate}
The sequence of natural numbers of crossings added alternately in steps (2) and (3) can be interpreted as the continued fraction expansion of a rational number $p/q\in \mathbb{Q}$ (with a minus sign in the case of negative rational tangles) and it is a well known fact that this is a complete invariant of 2-bridge links. Without loss of generality we only consider 2-bridge links that are closures of positive rational tangles. 
\begin{exa}
\[\text{The link associated to }\frac{7}{2}=[3,2]\text{ is } \quad \vcenter{\hbox{\includegraphics[height=1.5cm, angle=0]{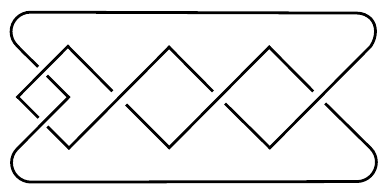}}} \]
\end{exa}
Our main result is:
\begin{thm}\label{mainthm} Let $L$ be a $2$-bridge link that is the closure of a positive rational tangle with continued fraction expansion $[a_r,\dots,a_1]$ with a total of $n$ crossings. Then the following hold:
\begin{enumerate}
\item There exists a sequence $(\tilde{P}_j(L)\in\mathbb{Z}[a^{\pm 1}, s^{\pm 1}](q) \mid j\in \mathbb{N})$, which only depends on the oriented link $L$ such that:  
\begin{align*} 
P_j(L) &= \tilde{P}_j(L)\mid_{s=1}P_j(\bigcirc) &\text{ if L is a knot.}\\
P_{i,j}(L) &= \tilde{P}_j(L)\mid_{s=q^{i-j}}P_i(\bigcirc) &\text{ if L has two components.}
\end{align*}
In particular, the colored HOMFLY polynomial of a two-component 2-bridge link, reduced with respect to color $i$, becomes independent of $i$ up to shifts in $q$-degree. 
\item $\tilde{P}_j(L)$ is given (up to multiplication by a monomial in $a$, $q$ and $s$) by the q-holonomic formula:
\[\tilde{P}_j(L) =  \sum_{i_1 = 0}^j\dots \sum_{i_{a_1} = i_{a_1-1}}^{j} \sum_{i_{a_1+1} = 0}^{i_{a_1}} \dots \sum_{i_{a_1+a_2} = 0 }^{i_{a_1+a_2-1}}\dots      \prod_{l=1}^n (-1)^{i_l} b_l(j,i_l, i_{l-1}) {n_l(j, i_k) \brack m_l(i_l, i_{l-1})} Cl[\partial L](i_n)  \]

This is a multi-dimensional sum with one index per crossing. Each index $i_l$ runs between $i_{l-1}$ and $j$ in the case of a top crossing and between $0$ and $i_{l-1}$ in the case of a right crossing. $b_l(j, i_l, i_{l-1})$ is a monic monomial in $q, a, s$ which depends only on $j$, $i_l$, $i_{l-1}$ and the boundary data of the tangle before the $l^{th}$ step of the inductive construction process. ${n_l(j, i_l) \brack m_l(i_l, i_{l-1})}$ is a $q$-binomial coefficient with $n_l(j, i_l)$ and $m_l(i_l, i_{l-1})$ depending at most linearly on $j$, $i_l$ and $i_{l-1}$. Finally $Cl[\partial L](i_n)$ is a quotient of products of $q$-Pochhammer symbols, which depends only on $j$, $i_n$ and the boundary data of the rational tangle before closing up.
\end{enumerate}

\end{thm}
Theorem \ref{mainthm} provides an explicit q-holonomic formula for the $r$-colored HOMFLY polynomials of $2$-bridge links. An implementation in Wolfram Mathematica can be found on the author's website \url{https://www.dpmms.cam.ac.uk/~pw360/}.

\begin{rem}
The HOMFLY polynomials $P_{r^t}(L)$ with respect to colorings with one-row partitions with $r$ boxes are related to $P_r(L)$ as follows:
\[P_{r^t}(L)(a,q) = (-1)^r P_r(L)(a,q^{-1})\]
This is proved, for example, in \cite{LP} Lemma 4.2.
\end{rem}

\subsection*{Acknowledgements}
This note is a follow-up to \cite{Wed}, where we compute categorified $sl(N)$ invariants of positive rational tangles and verify conjectures of Gukov and Sto\v{s}i\'c in the setting of tangles, see also \cite{GS}, \cite{GNSS}. I would like to thank Jacob Rasmussen for many interesting and fruitful discussions and Stavros Garoufalidis for his encouragement to write this note.
\footnote{The author's PhD studies at the Department of Pure Mathematics and Mathematical Statistics, University of Cambridge, are supported by the ERC grant ERC-2007-StG-205349 held by Ivan Smith and an EPSRC department doctoral training grant.}

\section{Proof of Theorem \ref{mainthm}}
(1) is well known for knots; in the more interesting case of links the statement holds for arbitrary colored links with an unknot component and is proved in section 5 of \cite{Wed}. We prove (2) in two steps. First we use the replacement rules of \cite{MOY} to evaluate the rational tangle associated to $[a_r,\dots, a_1]$ in an appropriate skein module and expand in a distinguished basis. This can be done by induction on the number of crossings and in the $l^{th}$ step we get a new one-parameter sum over $i_l$, a monomial $b_l(j, i_l, i_{l-1})$ and a q-binomial coefficient ${n_l(j, i_l) \brack m_k(i_l, i_{l-1})}$. Second, we replace basis elements in the expansion by the $i$-reduced MOY evaluations of their closures, this accounts for the factor $Cl[\partial L](i_n)$.

\subsection*{Step 1: evaluating a rational tangle}

Via the crossing replacement rules of \cite{MOY} a rational tangle can be interpreted as an element of the HOMFLY skein of a disc with four boundary points whose colors and orientations are determined by the tangle. Such a skein is a module over the base ring $\mathbb{Z}[a^{\pm 1}, s^{\pm 1}](q)$ spanned by planar, oriented, trivalent graphs with a flow on the edges, modulo local relations and boundary preserving isotopy, see e.g. \cite{Wed} section 2 for details. In the case of colors $i,j$ with $i\geq j$ each of these skeins is a free module of rank $j+1$ over the base ring with basis elements (indexed by $0\leq k\leq j$) as shown in Figure \ref{fig1}. The Figure also displays the four different closure operations and the linear operators $T$ and $R$ given by composing a skein element with a crossing on top or on the right. Throughout we assume $i\geq j$ and that the lower left strand is incoming and colored by $i$. 
 \begin{figure}[h]
\centerline{
\begindc{\commdiag}[4]
\obj(35,38)[1a]{$\langle UP[i,j,k]=\vcenter{\hbox{\includegraphics[height=1.4cm, angle=0]{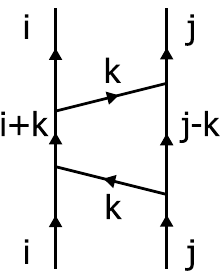}}}\rangle$}
\obj(85,38)[1b]{$\langle UPs[i,j,k]=\vcenter{\hbox{\includegraphics[height=1.4cm, angle=0]{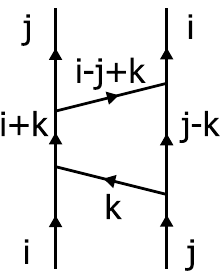}}}\rangle$}
\obj(20,19)[2a]{$\langle OP[i,j,k]= \vcenter{\hbox{\includegraphics[height=1.4cm, angle=0]{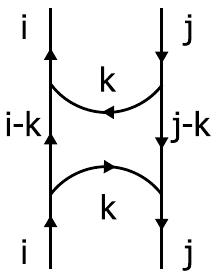}}}\rangle$}
\obj(35,0)[2b]{$\langle RIs[i,j,k]=\vcenter{\hbox{\includegraphics[height=1.4cm, angle=0]{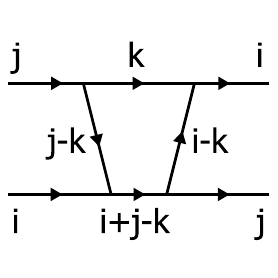}}}\rangle$}
\obj(85,0)[2c]{$\langle RI[i,j,k]=\vcenter{\hbox{\includegraphics[height=1.4cm, angle=0]{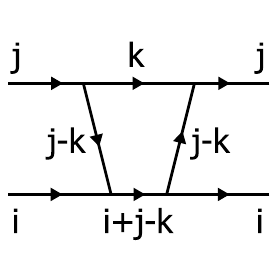}}}\rangle$}
\obj(100,19)[2d]{$\langle OPs[i,j,k]=\vcenter{\hbox{\includegraphics[height=1.4cm, angle=0]{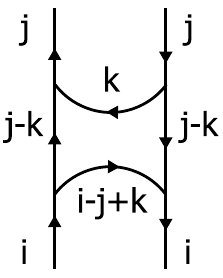}}}\rangle$}
\obj(60,19)[3]{$\mathbb{Z}[a^{\pm 1}, s^{\pm 1}](q)$}
\mor{1a}{1b}{$T$}[\atleft,\doubleopposite]
\mor{1a}{2a}{$R$}[\atright,\doubleopposite]
\mor{2a}{2b}{$T$}[\atright,\doubleopposite]
\mor{2b}{2c}{$R$}[\atright,\doubleopposite]
\mor{2c}{2d}{$T$}[\atright,\doubleopposite]
\mor{1b}{2d}{$R$}[\atleft,\doubleopposite]
\mor{1a}{3}{$Cl$}[\atleft,0]
\mor{2a}{3}{$Cl$}[\atleft,0]
\mor{2c}{3}{$Cl$}[\atleft,0]
\mor{2d}{3}{$Cl$}[\atleft,0]
\enddc
}
\caption{}
\label{fig1}
\end{figure}\\
The skein element of the rational tangle associated to $[a_r,\dots, a_1]$ is thus the image of the trivial tangle $UP[i,j,0]$ or $OP[i,j,0]$ (depending on the chosen orientation) under the linear operator $\cdots T^{a_3}R^{a_2}T^{a_1}$. The following Proposition describes the action of the operators $T$ and $R$ on the basis elements shown in Figure \ref{fig1}, up to overall multiplication by a monomial in $a$, $q$, and $s$. We write the first equation graphically and use more compact notation for the other eleven equations.
\begin{prop}
\label{twistrules}
\begin{enumerate}
\item $
\underbrace{\vcenter{\hbox{\includegraphics[height=1.4cm, angle=0]{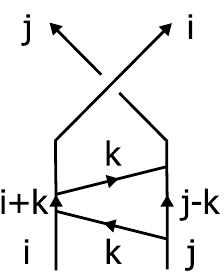}}}}_{TUP[i,j,k]}=\sum_{h=k}^j (-1)^{h} s^{k} q^{h(k+1)} {h \brack k} \underbrace{\vcenter{\hbox{\includegraphics[height=1.4cm, angle=0]{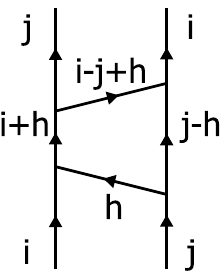}}}}_{UPs[i,j,h]} $
\item $TUPs[i,j,k]=\sum_{h=k}^j (-1)^{h} s^{h} q^{h(k+1)} {h \brack k} UP[i,j,h]$
\item $TOPs[i,j,k]=\sum_{h=k}^j (-1)^{h} a^k s^{h-k} q^{-2j k+h(k+1)} {h \brack k} RI[i,j,h] $
\item $TOP[i,j,k]=\sum_{h=k}^j (-1)^{h} a^k s^{-k} q^{-2j k+h(k+1)} {h \brack k} RIs[i,j,h] $
\item $TRI[i,j,k]=\sum_{h=k}^j (-1)^{h} a^h s^{k-h} q^{h(k+1-2j)} {h \brack k} OPs[i,j,h] $
\item $TRIs[i,j,k]=\sum_{h=k}^j (-1)^{h} a^h s^{-h} q^{h(k+1-2j)} {h \brack k} OP[i,j,h] $
\item $RUP[i,j,k]=\sum_{h=0}^k (-1)^{h} a^h s^{k-h}q^{j k+h(1-k-h)} {j-h \brack k-h} OP[i,j,h]$
\item $RUPs[i,j,k]=\sum_{h=0}^k (-1)^{h} a^h s^{-h}q^{j k+h(1-k-h)} {j-h \brack k-h} OPs[i,j,h]$
\item $ROP[i,j,k]=\sum_{h=0}^k (-1)^{h} a^k s^{h-k}q^{-j k + h(1+j-k)} {j-h \brack k-h} UP[i,j,h]$
\item $ROPs[i,j,k]=\sum_{h=0}^k (-1)^{h} a^k s^{-k}q^{-j k + h(1+j-k)} {j-h \brack k-h} UPs[i,j,h]$
\item $RRI[i,j,k]=\sum_{h=0}^k (-1)^{h} s^{k}q^{j k +h(1+j-k) } {j-h \brack k-h} RIs[i,j,h]$
\item $RRIs[i,j,k]=\sum_{h=0}^k (-1)^{h} s^{h}q^{j k + h(1+j-k)} {j-h \brack k-h} RI[i,j,h]$
\end{enumerate}

\end{prop}
\begin{proof}
Elementary \cite{MOY} skein theory, or decategorification ($t=-1$) of computations in section 3.3. of \cite{Wed}.
\end{proof}
Note that the $b_l(j, i_l=h,i_{l-1}=k)$, $n_l(j, i_l=h)$ and $m_l(i_l=h, i_{l-1}=k)$ can be read off from the formulas in Proposition \ref{twistrules}. Furthermore, the summation indices agree with the statement of Theorem \ref{mainthm}; top twists produce summations from $k$ up to $j$ and right twists from $0$ to $k$.

\subsection*{Step 2: closing off}

Suppose $L$ is a two-component 2-bridge link, then the invariant of the corresponding rational tangle is a linear combination of terms of one of four forms: $UP[i,j,k]$, $OP[i,j,k]$, $RI[i,j,k]$, $OPs[i,j,k]$. The first two can be closed off by connecting lower and upper endpoints (without introducing new crossings); the latter two are closed off by connecting left to right endpoints. The resulting graphs evaluate to the ground ring $\mathbb{Z}[a^{\pm 1}](q)$ as follows:

\begin{prop} \label{closeprop} We introduce the notation ${m \brack n }_a:= \frac{q^{n(n-m)}}{a^n} \frac{((aq^m)^2,q^{-2})_n}{(q^2,q^2)_n}= \prod_{l=0}^{n-1}\frac{a q^{m-l}-a^{-1} q^{l-m}}{q^{l+1}-q^{-l-1}} $ so that, in particular, $P_n(\bigcirc)= {0\brack n}_a$. Then the closures of $UP[i,j,k]$, $OP[i,j,k]$, $RI[i,j,k]$ and $OPs[i,j,k]$ evaluate as follows:
\begin{align*}
ClUP[i,j,k] &= {-k \brack j-k }_a {-i \brack k }_a P_i(\bigcirc) = {-k \brack j-k }_a (\prod_{l=0}^{k-1}\frac{a s^{-1} q^{-j-l}-a^{-1} s q^{l+j}}{q^{l+1}-q^{-l-1}})P_i(\bigcirc)\mid_{s=q^{i-j}}  \\
 ClOP[i,j,k] &= {-k \brack j-k }_a {i \brack k} P_i(\bigcirc) ={-k \brack j-k }_a (\prod_{l=0}^{k-1}\frac{ s q^{j-l}- s^{-1} q^{l-j}}{q^{l+1}-q^{-l-1}}) P_i(\bigcirc)\mid_{s=q^{i-j}}   \\ 
 ClRI[i,j,k] &= {k-j \brack k }_a {-i \brack j-k }_a P_i(\bigcirc) = {k-j \brack k }_a (\prod_{l=0}^{j-k-1}\frac{a s^{-1} q^{-j-l}-a^{-1} s q^{l+j}}{q^{l+1}-q^{-l-1}}) P_i(\bigcirc)\mid_{s=q^{i-j}}  \\
  ClOPs[i,j,k] &= {k-j \brack k }_a {i \brack j-k} P_i(\bigcirc)={k-j \brack k }_a (\prod_{l=0}^{j-k-1}\frac{ s q^{j-l}- s^{-1} q^{l-j}}{q^{l+1}-q^{-l-1}}) P_i(\bigcirc)\mid_{s=q^{i-j}}   
 \end{align*}
\end{prop}
\begin{proof}
Elementary \cite{MOY} skein theory.
\end{proof}
Note that $Cl[\partial L](i_n) := \frac{ClX[i,j,i_n]}{P_i(\bigcirc)} \in \mathbb{Z}[a^{\pm 1}, s^{\pm 1}](q)$ can  easily be read off from the right hand sides of the above equations (here we have written $X$ for the appropriate type of basis element), and as a function of $s$ $Cl[\partial L](i_n)$ is independent of $i$. This completes the proof of Theorem \ref{mainthm} in the case of two-component 2-bridge links.\\

Let $L$ be a 2-bridge knot. Then we can still use the complex for the rational tangle as above, but have to set $i=j$ and hence $s=1$. In this case $UPs[j,j,k]=UP[j,j,k]$ and $RIs[j,j,k]=RI[j,j,k]$ which allows to compute the evaluation of the closure of these graphs via the formulas in Proposition \ref{closeprop}.

\end{document}